\documentclass[12pt,letterpaper]{amsart}

\usepackage{epsfig,amsmath,amsthm,amssymb,graphicx}
\usepackage[top=1in, bottom=1in, left=1in, right=1in]{geometry}
\usepackage{color}
\usepackage{pb-diagram}

\newtheorem*{rep@theorem}{\rep@title}
\newcommand{\newreptheorem}[2]{%
\newenvironment{rep#1}[1]{%
\def\rep@title{#2 \ref{##1}}%
\begin{rep@theorem}}%
{\end{rep@theorem}}}

\newreptheorem{theorem}{Theorem}
\newreptheorem{lemma}{Lemma}
\newreptheorem{corollary}{Corollary}
\newreptheorem{proposition}{Proposition}

\newtheorem{prop}{Proposition}
\newtheorem{theorem}{Theorem}
\newtheorem{corollary}{Corollary}
\newtheorem{lemma}{Lemma}
\theoremstyle{definition}

\newcommand{\B}{\mathcal{B}}

\newcommand{\Z}{\mathbb{Z}}
\newcommand{\C}{\mathcal{C}}
\newcommand{\x}{\times}
\newcommand{\link}{L=K_1~\cup~\dots~\cup~K_m}

\newcommand{\I}{[0,1]}
\newcommand{\bd}{\partial}
\newcommand{\F}{\mathcal{F}}

\newcommand{\inv}{^{-1}}

\begin{document}
\title{Splittings of Link Concordance Groups}

\author{Taylor E. Martin}
\address{Department of Mathematics, Sam Houston University}
\email{taylor.martin@shsu.edu}

\author{Carolyn A. Otto}
\address{Department of Mathematics, University of Wisconsin--Eau Claire}
\email{ottoa@uwec.edu}

\date{\today}


\begin{abstract}
We establish several results about two short exact sequences involving lower terms of the $n$-solvable filtration, $\{\F^m_n\}$ of the string link concordance group $C^m$. We utilize the Thom-Pontryagin construction to show that the Sato-Levine invariants $\bar{\mu}_{(iijj)}$ must vanish for 0.5-solvable links. Using this result, we show that the short exact sequence $0\rightarrow \F^m_0/\F^m_{0.5} \rightarrow \F^m_{-0.5}/\F^m_{0.5} \rightarrow \F^m_{-0.5}/\F^m_0 \rightarrow 0$ does not split for links of two or more components, in contrast to the fact that it splits for knots.  Considering lower terms of the filtration $\{\F^m_n\}$ in the short exact sequence $0\rightarrow \F^m_{-0.5}/\F^m_{0} \rightarrow \C^m/\F^m_{0} \rightarrow \C^m/\F^m_{-0.5} \rightarrow 0$, we show that while the sequence does not split for $m\ge 3$, it does indeed split for $m=2$.  This allows us to determine that the quotient $\C^2/\F^2_0 \cong \Z_2\oplus \Z_2\oplus\Z_2 \oplus \Z$.

\end{abstract}

\maketitle

\section{Introduction}
 


 In 1966, Fox and Milnor showed that concordance classes of knots form an abelian group $\C$, known as the knot concordance group.  To investigate the structure of this group, Cochran, Orr and Teichner defined in \cite{COT} the $n$-solvable filtration,$\{\mathcal{F}_n\}$,
$$\{0\} \subset \cdots \subset \mathcal{F}_{n+1} \subset \mathcal{F}_{n.5} \subset \mathcal{F}_n \subset \cdots \subset \mathcal{F}_{0.5} \subset \mathcal{F}_0 \subset \mathcal{C}.$$  

In this paper, we focus our attention to the more general (string) link concordance group, $\C^m$. The $n$-solvable filtration given by Cochran, Orr, and Teichner can be similarly defined to give the $n$-solvable filtration of $\C^m$. A successful tool in studying the structure of $\C^m$ has been in analyzing the successive quotients of the this filtration. For example, Harvey showed in ~\cite{H} that $\F^m_n/\F^m_{n+1}$ is a nontrivial group that contains an infinitely generated subgroup generated by boundary links.  Harvey and Cochran generalized this result, showing that $\F^m_n/\F^m_{n.5}$ also contains an infinitely generated subgroup ~\cite{CH}.  The second author showed the quotients $ \F^m_{n.5}/\F^m_{n+1}$ are also nontrivial by showing that they contain an infinite cyclic subgroup generated by iterated Bing doubles, \cite{Otto}. Cha, Harvey, Cochran, and Leidy have given significant contributions to this area as found in ~\cite{CHL}, ~\cite{Cha}. The first author gave a characterization of 0-solvable links, $\F^m_0$ in ~\cite{Martin} that will be of significant use in this paper.

The main objective of this paper is to investigate lower order quotients of the $n$-solvable filtration of $\C^m$ in order to learn more about their structure. In order to achieve this, we consider several short exact sequences and determine the conditions on the number of components of a link that will result in a splitting of these sequences.

In section \ref{section:SES1}, we consider the short exact sequence  $0\rightarrow \F^m_0/\F^m_{0.5} \rightarrow \F^m_{-0.5}/\F^m_{0.5} \rightarrow \F^m_{-0.5}/\F^m_0 \rightarrow 0$, where the notation $\F^m_{-0.5}$ refers to the class of $m$ component links with vanishing pairwise linking numbers.  While it is known this sequence splits for knots, we show that it does not split for links with two or more components.

\begin{reptheorem}{thm:SES1}The short exact sequence $0\rightarrow \F^m_0/\F^m_{0.5} \rightarrow \F^m_{-0.5}/\F^m_{0.5} \rightarrow \F^m_{-0.5}/\F^m_0 \rightarrow 0$ does not split for $m \geq 2$.
\end{reptheorem}

In section \ref{section:SES2}, we consider another short exact sequence with slightly lower order terms of the $n$-solvable filtration, $0\rightarrow \F^m_{-0.5}/\F^m_{0} \rightarrow \C^m/\F^m_{0} \rightarrow \C^m/\F^m_{-0.5} \rightarrow 0$.  When considering $m=2$, we show that this sequence splits. 

\begin{repproposition}{prop:SES2}
The short exact sequence $0\rightarrow \F^2_{-0.5}/\F^2_{0} 
{\rightarrow} \C^2/\F^2_{0} 
{\rightarrow} \C^2/\F^2_{-0.5} \rightarrow 0$ splits.
\end{repproposition}

By showing the above sequence splits, we can write $\C^2/\F^2_{0}$ as a direct product.  We then can use this splitting to prove that $\C^2/\F^2_{0}$ is an abelian group by considering commutators in this quotient. Moreover, we obtain the following corollary characterizing $\C^2/\F^2_0$.

\begin{reptheorem}{thm:ab}
The quotient $\C^2/\F^2_0 $ is abelian.
\end{reptheorem}

\begin{repcorollary}{corollary:structure}
The quotient $\C^2/\F^2_0 \cong \Z_2 \oplus \Z_2 \oplus \Z_2 \oplus \Z$.
\end{repcorollary}

Finally we demonstrate that when links have three or more components, the short exact sequence does not split and the structure of $\C^m/\F^m_0$ remains not fully known.

\begin{reptheorem}{thm:SES2nosplit}
For $m \ge 3$, the short exact sequence $0\rightarrow \F^m_{-0.5}/\F^m_{0} 
{\rightarrow} \C^m/\F^m_{0} 
{\rightarrow} \C^m/\F^m_{-0.5} \rightarrow 0$ does not split.
\end{reptheorem}

\section{Preliminaries}

An \textit{$m$-component link} is an embedding $\coprod_m S^1 \hookrightarrow S^3$. In order to define a group structure on concordance classes of links, it is necessary to define string links. Let $D$ be the unit disk, $I$ the unit interval and $\{p_1, p_2, \ldots, p_m\}$ be $m$ points in the interior of $D$.  An \textit{$m$-component string link} is a smooth proper embedding $\sigma : \coprod_{i=1}^m I_i \rightarrow D \times I$ such that $\sigma |_{I_i}(0)=\{p_i\} \times \{0\}$ and $\sigma |_{I_i}(1)=\{p_i\} \times \{1\}.$   
Two $m$-component string links $\sigma_1$ and $ \sigma_2$ are \textit{concordant} if there exists a smooth embedding
$H:\coprod_m (I \times I) \rightarrow B^3 \times I$ that is transverse to the boundary and such that
$H|_{\coprod_m I\times \{0\}}=\sigma_1$, $H|_{\coprod_m I\times \{1\}}=\sigma_2$, and $H|_{\coprod_m \partial I\times I}=j_0 \times id_I$ where $j_0: \coprod_m \partial I \rightarrow S^2$.

The concordance classes of $m$-component string links form a group, denoted $\mathcal{C}^m$, and is known as the \textit{string link concordance group} under the operation of stacking.  The identity class of this group is the class of slice string links, those concordant to the trivial string link. The inverses are the string links obtained by reflecting the string link about $D \times \{1/2\}$ and reversing the orientation.  When $m=1$, $\mathcal{C}^m$ is the knot concordance group. Every link has a string link representative. In other words, given any link, $L$ in $S^3$, there exists a string link $\sigma$ such that $\hat \sigma$ is isotopic to $L$~\cite{HL1}.

Cochran, Orr, and Teichner \cite{COT} defined the condition of $n$-solvability for links. An ordered, oriented $m$-component link $L$ is $n$\textit{-solvable} if the zero-framed surgery, $M_L$, bounds a compact, smooth 4-manifold, $W^4$, such that the following conditions hold:
\begin{itemize}
\item[i)]  $H_1(M_L; \mathbb{Z}) \cong \mathbb{Z}^m$ and $H_1(M_L) \rightarrow H_1(W; \mathbb{Z})$ is an isomorphism induced by the inclusion map;
\item[ii)] $H_2(W; \mathbb{Z})$ has a basis consisting of connected, compact, oriented surfaces, $\{L_i, D_i\}_{i=1}^r$, embedded in $W$ with trivial normal bundles, where the surfaces are pairwise disjoint except that, for each $i$, $L_i$ intersects $D_i$ transversely once with positive sign;
\item[iii)] For all $i\text{, } \pi_1(L_i) \subset \pi_1(W)^{(n)}$ and  $\pi_1(D_i) \subset \pi_1(W)^{(n)}$ where $\pi_1(W)^{(n)}$ is the $n^{th}$ term of the derived series.  The derived series of a group $G$, denoted $G^{(n)}$ is defined recursively by $G^{(0)}:=G$ and $G^{(i)}:=[G^{(i-1)},G^{(i-1)}]$.
\end{itemize}

The 4-manifold $W$ is called an $n$-solution for L. We consider a string link to be $n$-solvable if its closure is an $n$-solvable link. 
A link is $n.5$-solvable if, in addition to the above definition, $ \pi_1(L_i) \subset \pi_1(W)^{(n+1)}$ for all $i$.  In this instance, $W$ is called an $n.5$-solution for $L$.

The notation $\F^m_n$ is used to denote the set of concordance classes of $n$-solvable $m$-component links. For convenience, we denote $\F^m_{-0.5}$ to be the set of $m$ component links with vanishing pairwise linking numbers. The notion of $n$-solvability for string links gives rise to the $n$-solvable filtration, $\{\mathcal{F}^m_n\}$, of $\C^m$. 
$$\{0\} \subset \cdots \subset \mathcal{F}^m_{n+1} \subset \mathcal{F}^m_{n.5} \subset \mathcal{F}^m_n \subset \cdots \subset \mathcal{F}^m_{0.5} \subset \mathcal{F}^m_0 \subset \F^m_{-0.5} \subset \mathcal{C}^m.$$

In the 1950’s, Milnor defined a classical family of link invariants called $\bar \mu$̄-invariants ~\cite{Mil1}, ~\cite{Mil2}. Let $L$ be an ordered, oriented $m$-component link.  Milnor invariants are denoted $\bar \mu_{(I)}(L)$ where $I=i_1i_2\dots i_k$ is a word composed of integers: $i_j \in \{1, 2, \dots, m\}$.  The index $i_j$ refers to the $j^{th}$ component of $L$.  We say that both $I$ and $\bar \mu_{(I)}(L)$ have length $k$.

There is some indeterminacy of these invariants due to the fact that there is a choice for the meridians of the link.  Thus these invariants are not true link invariants.  However, Habegger and Lin showed that this indeterminacy corresponds to exactly to the choice of ways of representing a link as a closure of a string link ~\cite{HL1}.  Thus Milnor's invariants are string link invariants and are also concordance invariants ~\cite{Casson}.

In this paper will we often utilize Milnor's $\bar \mu$-invariants, specifically the Sato-Levine invariant, $\bar \mu_{(iijj)}(L)$ and the triple linking number, $\bar \mu_{(ijk)}(L)$. There are many ways of defining Milnor’s invariants, but these specific two invariants can be thought of as higher order cup products and we give their definitions geometrically.  Let $L=K_1 \cup K_2 \cup \cdots \cup K_m$.  Every 2-component sublink $K_i \cup K_j$ gives a Sato-Levine invariant of $\bar \mu_{(iijj)}(L)$ of $L$. We can compute this invariant by considering oriented Seifert surfaces $\Sigma_i$ and $\Sigma_j$ for $K_i$ and $K_j$ in the exterior of the link. We may choose these surfaces in such a way that $\Sigma_i \cap \Sigma_j=\gamma \cong S^1$ ~\cite{Tim}. Then, we push the curve $\gamma$ off of one of the surfaces $\Sigma_i$ in the positive normal direction to obtain a new curve $\gamma^+$. We define the Sato-Levine invariant to be $lk(\gamma, \gamma^+)$. This is a well-defined link invariant provided that the pairwise linking numbers of $L$ all vanish, and it is equal to the Milnor's invariant $\bar \mu_{(iijj)}(L)$ ~\cite{Tim}.

To compute the Milnor's invariant $\bar \mu_{(ijk)}(L)$, we first consider the sublink $J = K_i \cup K_j \cup K_k$ of the link $L$.  Let $\Sigma_i$,$\Sigma_j$, and $\Sigma_k$ be oriented Seifert surfaces in the link exterior for $K_i$, $K_j$, and $K_k$, respectively. The intersection $\Sigma_i\cap \Sigma_j \cap \Sigma_k$ is a collection of points which are given an orientation induced by the outward normal on each Seifert surface. The count of these points up to sign is well-defined when the pairwise linking numbers of $L$ vanish; this count gives us $\bar \mu_{(ijk)}(L)$ ~\cite{Tim}.

\section{The sequence $0\rightarrow \F^m_0/\F^m_{0.5} \rightarrow \F^m_{-0.5}/\F^m_{0.5} \rightarrow \F^m_{-0.5}/\F^m_0 \rightarrow 0$}\label{section:SES1}

We will begin by considering the following short exact sequence of quotients of the $n$-solvable filtration. 
 \begin{equation}\label{eq:SES1}
0\rightarrow \F^m_0/\F^m_{0.5} \stackrel{f}{\rightarrow} \F^m_{-0.5}/\F^m_{0.5} \stackrel{g}{\rightarrow} \F^m_{-0.5}/\F^m_0 \rightarrow 0.
\end{equation} 

It is known that this sequence splits in the knot concordance group; we recreate the argument here. Recall that $\F^m_{-0.5}$ denotes the class of $m$ component links with vanishing pairwise linking numbers. In the knot concordance group, we have that $\F_{-0.5}$  is itself the whole group $\C$. Thus, here we consider the following sequence.
\begin{displaymath}
0\rightarrow \mathcal{F}_0/\mathcal{F}_{0.5} \rightarrow \mathcal{C}/\mathcal{F}_{0.5} \stackrel{g}{\rightarrow} \mathcal{C}/\mathcal{F}_0 \rightarrow 0.
\end{displaymath}
 
This sequence splits as we can define a section $s:\C/\F_0 \rightarrow \C/\F_{0.5}$ such that $g \circ s=id$. The quotient $\C/\F_0 \cong \Z_2$ is given by the Arf invariant and generated by the figure eight knot $4_1$, with Arf invariant 1. The subgroup $\F_{0.5}$ is the subgroup of algebraically slice knots \cite{COT}. As the figure eight knot is amphichiral, it has order two in the knot concordance group. Thus, the inclusion map gives a right splitting of the short exact sequence. 

When considering the sequence (\ref{eq:SES1}) for links of two or more components, we show that no analogous section exists. 

\begin{theorem}\label{thm:SES1}
The short exact sequence $0\rightarrow \F^m_0/\F^m_{0.5} \rightarrow \F^m_{-0.5}/\F^m_{0.5} \rightarrow \F^m_{-0.5}/\F^m_0 \rightarrow 0$ does not split for $m \geq 2$.
\end{theorem}

To prove Theorem \ref{thm:SES1}, we rely on the following lemma showing that the Sato-Levine invariants of a $0.5$-solvable link vanish. 

\begin{lemma}\label{05solvable}
For an ordered, oriented, $m$-component, 0.5-solvable link $\link$, and any $1\le i<j\le m$, the Sato-Levine invariant $\bar{\mu}_{(iijj)}(L) = 0$.
\end{lemma}

\begin{proof}
Let $J=K_i \cup K_j$ be a two-component sublink of $L$. Then, $J$ is 0.5-solvable; we wish to show that $\bar{\mu}_{(1122)}(J)=0$. Let $W$ be a 0.5-solution for $J$ and let $\hat{W}$ be the closure of $W$ obtained from attaching 0-framed 2-handles along the meridians $\mu_1$ and $\mu_2$ in $M_J = \bd W$ and then attaching a 4-handle $\B^4$. We can choose oriented Seifert surfaces $\Sigma_1$ and $\Sigma_2$ for $K_1$ and $K_2$ such that $\Sigma_1 \cap \Sigma_2 = \gamma \cong S^1$ \cite{Tim}. Then, $\bar{\mu}_{(1122)}(J) = lk(\gamma, \gamma^+)$ \cite{Tim}. Let $\hat{\Sigma}_i$ be the closure of $\Sigma_i$ in $M_J$. 

Using the Thom-Pontryagin construction, we define maps $f_1:~M_J~\rightarrow~S^1$ and $f_2:~M_J~\rightarrow~S^1$ in the following way. Let $\hat{\Sigma}_i \x [-1,1]$ be a product neighborhood of $\hat{\Sigma}_i$, where the $+1$ corresponds to the positive side of the surface $\hat{\Sigma}_i$. Define $f_i: \hat{\Sigma}_i \x [-1,1] \rightarrow S^1$ by $f_i((x,t)) = e^{\pi it}$, and for $y \in M_J - (\hat{\Sigma}_i \x [-1,1])$, we let $f_i(y)=-1$. We then consider the map $f = f_1 \x f_2: M_J \rightarrow S^1 \x S^1$. The maps $\pi_1$ and $\pi_2$ are the standard projection maps, so that $f_i = \pi_i \circ f$.  

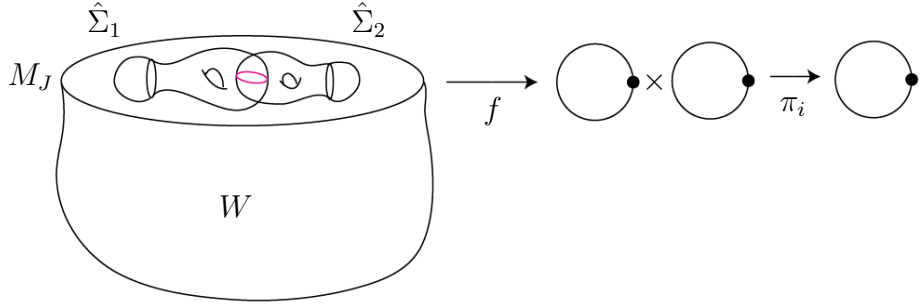
\begin{figure}
\centering
\includegraphics[scale=.7]{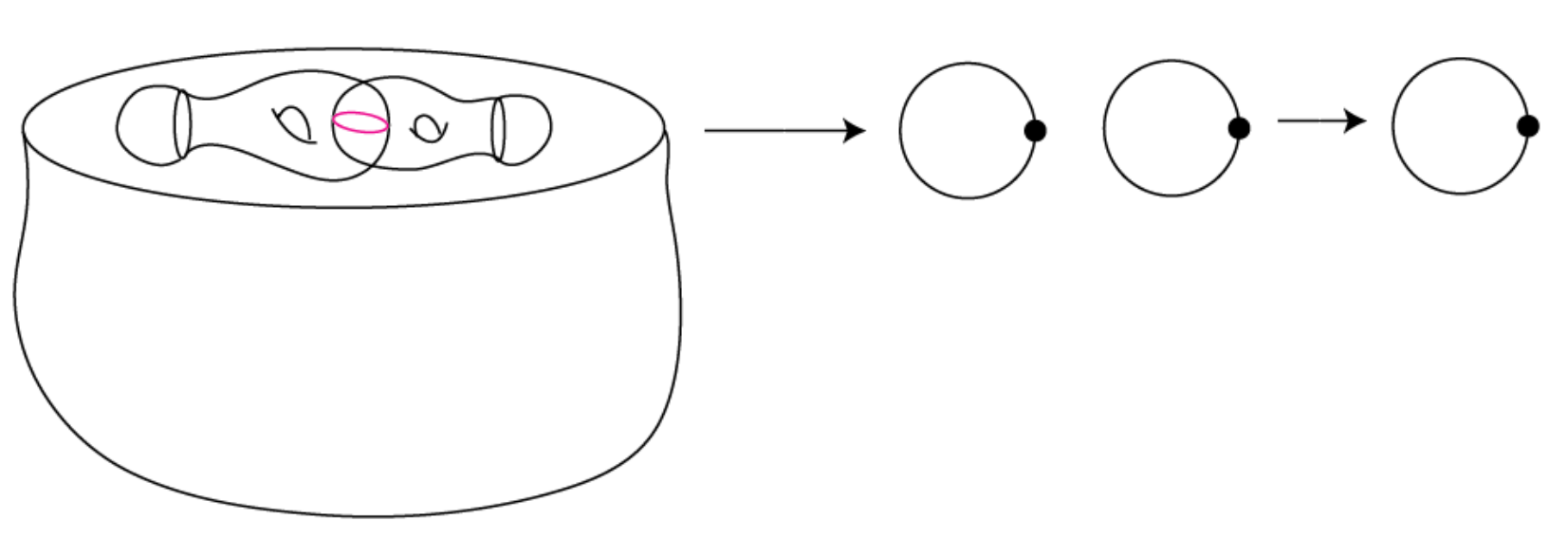}
\put(-170,77){$f$}
\put(-58,80){$\pi_i$}
\put(-270,40){$W$}
\put(-320,110){$\hat{\Sigma}_1$}
\put(-220,110){$\hat{\Sigma}_2$}
\put(-350,90){$M_J$}
\put(-110,88){$\x$}
\caption{The Thom-Pontryagin construction on $W$}
\end{figure}

The map $f:M_J \rightarrow S^1 \x S^1$ induces an isomorphism on first homology, $f_*: H_1(M_J) \rightarrow H_1(S^1 \x S^1)$. As $W$ is a 0.5-solution, the inclusion map $i_*$ is also an isomorphism. 

\[
\begin{diagram}
\node{\pi_1(M_J)} \arrow{e}
\node{H_1(M_J)} \arrow{s,lr}{\cong}{i_*} \arrow{ese,tb}{f_*}{\cong}
\\
\node{\pi_1(W)} \arrow{e}
\node{H_1(W)} \arrow[2]{e,t,3,..}{\alpha}
   \node[2]{H_1(S^1 \x S^1)}
\end{diagram}
\]

As in the above diagram, define the map $\alpha: H_1(W) \rightarrow H_1(S^1 \x S^1)$. Because $\pi_1(S^1 \x S^1) \cong \Z^2$ is abelian, we can extend $\alpha$ to the map $\bar{\alpha}: \pi_1(W) \rightarrow \pi_1(S^1 \x S^1)$, as in the following diagram. 

\[
\begin{diagram}
\node{\pi_1(W)}\arrow{e,t}{\bar{\alpha}}\arrow{s}
\node{\pi_1(S^1 \x S^1)}\arrow{s,r}{\cong}
\\
\node{H_1(W)}\arrow{e,tb}{\alpha}{\cong}
\node{H_1(S^1 \x S^1)}
\end{diagram}
\]

We wish to show the existence of an extension $\bar{f}: W \rightarrow S^1 \x S^1$ such that $\bar{f}\vert_{\bd W} = f$ and $\bar{f}_* = \bar{\alpha}$. Consider the following commutative diagram:

\[
\begin{diagram}
\node{}
\node{\pi_1(\bd W)}\arrow{sw,t}{i_*}\arrow{se,t}{f_*}\arrow[2]{s}
\\
\node{\pi_1(W)}\arrow{e,-}\arrow[2]{s} 
\node{}\arrow{e,t}{\bar{\alpha}}
\node{\pi_1(S^1 \x S^1)}\arrow[2]{s,r}{\cong}
\\
\node{}
\node{H_1(\bd W)}\arrow{sw,tb}{\cong}{i_*}\arrow{se,tb}{f_*}{\cong}
\\
\node{H_1(W)}\arrow[2]{e,tb}{\cong}{\alpha}
\node{}
\node{H_1(S^1\x S^1)}
\\
\end{diagram}
\]

Since the CW-complex $S^1 \x S^1$ is an Eilenberg-Maclane space $K(\Z^2, 1)$, we can extend the map $f$ to a map $\bar{f}:W \rightarrow S^1 \x S^1$. The pre-image $\bar{f}^{-1}_i (1) = M_i$ is a 3-submanifold of $W$ such that $\bd M_i = \hat{\Sigma}_i$. Moreover, $M_1 \cap M_2 = \bar{f}^{-1}((1,1)) = F$ is a surface $F \subset W$ such that $\bd F = \gamma$. Furthermore, the curve $\gamma$ bounds a surface $S$ in the attached 4-handle $\B^4 \subset \hat{W}$, so we let $\hat{F} = F \cup_{\gamma} S$ be a closed surface in $\hat{W}$. 

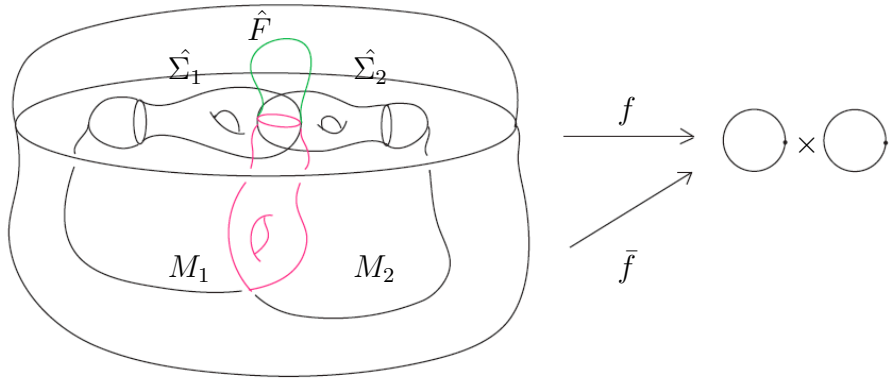
\begin{figure}[h!]
\centering
\includegraphics[scale=.4]{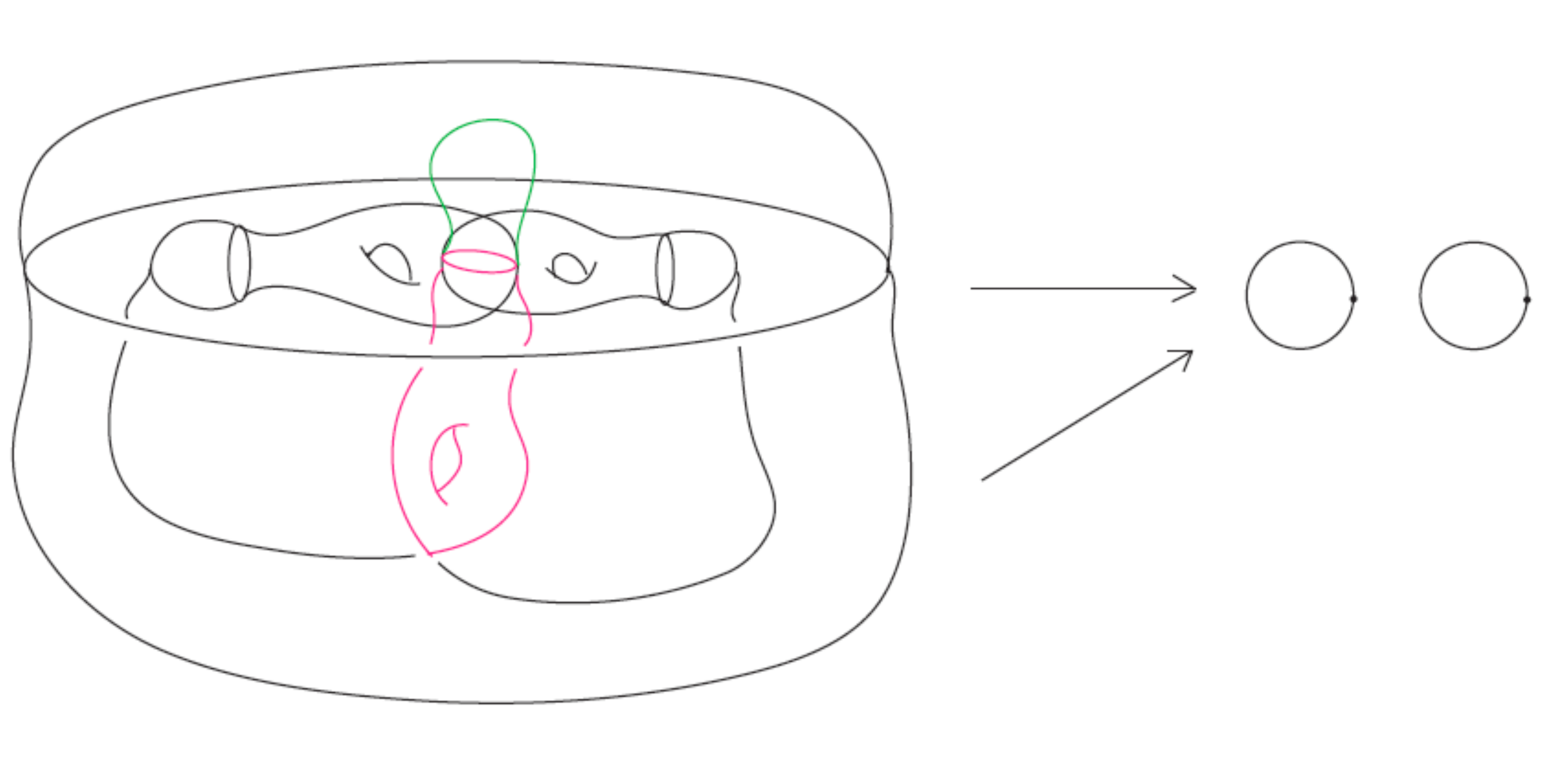}
\put(-280,125){$\hat{\Sigma_1}$}
\put(-210,125){$\hat{\Sigma_2}$}
\put(-280,50){$M_1$}
\put(-210,50){$M_2$}
\put(-250,140){$\hat{F}$}
\put(-110,110){$f$}
\put(-110,50){$\bar{f}$}
\put(-43,97){$\x$}
\caption{The surface $\hat{F}$}
\label{fig:05solv}
\end{figure}

In $S^1 \x S^1$, we consider $D_p$, an $\epsilon$-neighborhood of the point $p=(1,1)$. Let $q=(e^{2\pi i \epsilon},1)$ be a point on the boundary of $D_p$. Then, $\bar{f}^{-1}(q)$ is a surface $F^+$ in $W$ with boundary $\gamma^+$, a push-off of the curve $\gamma$ in $M_J$. The surfaces $F$ and $F^+$ are disjoint, as they map to distinct points in $S^1 \x S^1$. Furthermore, if we consider a path $\alpha (t) = (e^{2 \pi i \epsilon t},1), t \in \I$, the pre-image $\bar{f}^{-1}(\alpha (t))$ is a 3-submanifold of $W$ that is cobound by $F$ and $F^+$. We may also consider the surface $\hat{F^+} \subset \hat{W}$, a closure of $F^+$ given by attaching a surface $T$ in the 4-handle to $F^+$ along the curve $\gamma^+$. Then, $\bar{\mu}_{(1122)}(J)= lk(\gamma, \gamma^+) = \hat{F} \cdot \hat{F^+} = S \cdot T$, as surfaces $\hat{F}$ and $\hat{F^+}$ can only intersect in their caps in $\B^4$. We will use the following proposition to complete the proof.

\begin{prop}
For $\alpha$ a simple closed curve in $W$, the homology class $[\alpha] \in H_1(W)$ is given by the pair $(m_1,m_2) \in \Z \oplus \Z$ where $m_i = \alpha \cdot M_i$.
\end{prop}

\begin{proof}
$H_1(W) = \langle i_*(\mu_1),i_*(\mu_2) \rangle$ is generated by the inclusion of the meridians of the link components of $J$. Thus, the homology class $[\alpha]$ can be written as $[\alpha]=m_1[\mu_1] + m_2[\mu_2] = (m_1,m_2)$. Recalling the Thom-Pontryagin construction, $\hat{\Sigma}_i = f_i^{-1}(1)$ and $M_i = \bar{f}_i^{-1}(1)$. We then see that the intersection $\mu_i \cdot M_j = \delta_{ij}$. Therefore, $\alpha \cdot M_i  = m_i$. This says that the class $[\alpha] \in H_1(W)$ is given by the pair $(m_1,m_2)$ such that $m_i = \alpha \cdot M_i$. 
\end{proof}

Now, we consider the surfaces $\{L_i,D_i\}_{i=1}^r$ that generate $H_2(W)$, where $\pi_1(L_i) \subseteq \pi_1(W)^{(1)} = [\pi_1(W),\pi_1(W)]$. Choose the surfaces $L_i$ to be transverse to $M_1$ and $M_2$. We may assume that $\{L_i\} \subset int(W)$. For each $i$, consider the intersection $L_i \cap F = L_i \cap M_1 \cap M_2$. The intersection $L_i \cap M_1 = b_i$, where $b_i$ is some circle(s) on surface $L_i$. Because $\pi_1(L_i) \subseteq \pi_1(W)^{(1)}$, it must be true that $[b_i]=0 \in H_1(W)$. By the proposition, we then must have that $b_i \cdot M_2 = 0$ for each $i$. This tells us that $L_i \cdot \hat{F}=0$ for each $i$. We note that, as $F$ and $F^+$ cobound a 3-manifold in $W$, $L_i \cdot \hat{F}^+ = 0$ as well. 


We then can write the surfaces $\hat{F}$ and $\hat{F^+}$ in terms of the generators $\{L_i,D_i\}$ of $H_2(\hat{W})$. Let $\hat{F}= \sum_{i=1}^r{(a_iL_i+b_iD_i)}$ and let $\hat{F^+} = \sum_{i=1}^r{(a_i^+L_i + b_i^+D_i)}$. For each $i$, $0 = L_i \cdot \hat{F} = L_i \cdot \sum_{i=1}^r{(a_iL_i+b_iD_i)} = b_i $ and $0 = L_i \cdot \hat{F}^+ = L_i \cdot \sum_{i=1}^r{(a_i^+L_i + b_i^+D_i)} = b_i^+$. Thus, all of the $b_i$'s = 0 and all of the $b_i^+$'s = 0. Finally, we can conclude that $\bar{\mu}_{(1122)}(J) = \hat{F} \cdot \hat{F^+} = (\sum_{i=1}^r a_iL_i)\cdot(\sum_{i=1}^r a_i^+L_i)$. Recalling that $L_i \cdot L_j = 0$, this tells us that $\bar{\mu}_{(1122)}(J) = 0$. Therefore, if link $\link$ is 0.5-solvable, its Sato-Levine invariants $\bar{\mu}_{(iijj)}(L)$ all vanish.
\end{proof}

We now proceed with the proof of Theorem \ref{thm:SES1}.
\begin{proof}[Proof of Theorem \ref{thm:SES1}.]
 Consider the following short exact sequence,
\begin{equation}
0\rightarrow \F_0^m/\F_{0.5}^m \stackrel{f}{\rightarrow} \F_{-0.5}^m/\F_{0.5}^m \stackrel{g}{\rightarrow} \F_{-0.5}^m/\F_0^m \rightarrow 0.
\end{equation}
Suppose that this sequence does indeed split and that we have a section $s:\F^m_{-0.5}/\F^m_0 \rightarrow \F^m_{-0.5}/\F^m_{0.5}$ such that $g \circ s = id$. The first author showed in \cite{Martin} that $\F^m_{-0.5}/\F^m_0 \cong \Z^{m \choose 3}\oplus \Z_2^{m \choose 2} \oplus \Z_2^m$. We consider an element of order 2 in this group. Let $L$, as in Figure \ref{fig:WH}, be a string link whose first two components form the Whitehead link and whose remaining $m-2$ components are trivial.  Then $s(L)$ must be of order two in $\F^m_{-0.5}/\F^m_{0.5}$, and so $\bar{\mu}_{(1122)}(2s(L)) = 2 \bar{\mu}_{(1122)}(s(L)) = 0$. Thus, $\bar{\mu}_{(1122)}(s(L))=0$. However, $g(s(L))=J$, where $J$ is an $m$-component string link concordant to $L$, the Whitehead string link. However, $\bar{\mu}_{(1122)}(J) = \bar{\mu}_{(1122)}(s(L))=0$ while $\bar{\mu}_{(1122)}(L)=1$. Thus, such a section cannot exist, and hence the short exact sequence does not split.

\begin{figure}
\centering
\includegraphics[scale=.4]{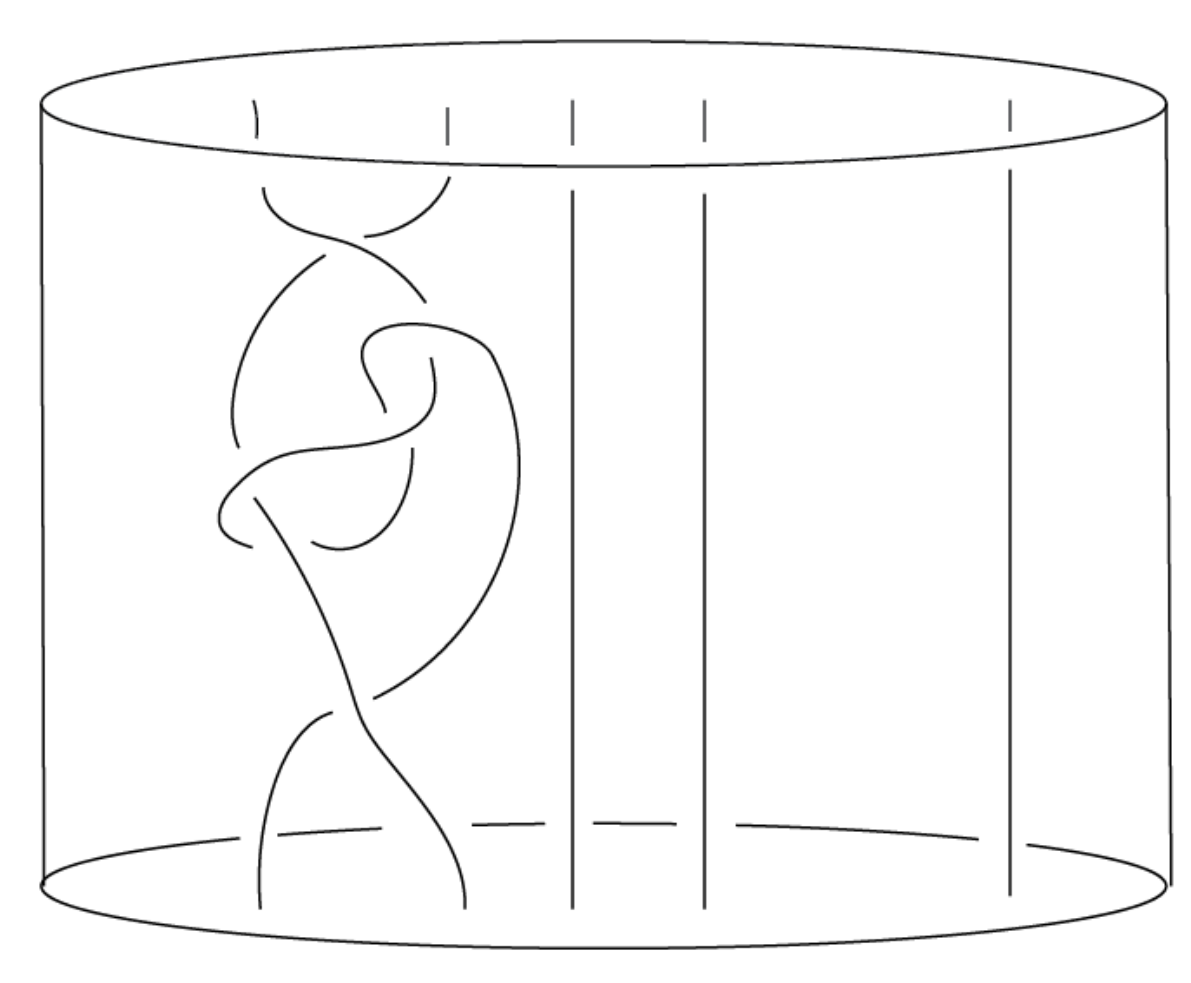}
\caption{Whitehead String Link}
\label{fig:WH}
\end{figure}
\end{proof}


\section{The sequence $0\rightarrow \F^m_{-0.5}/\F^m_{0} \rightarrow \C^m/\F^m_{0} \rightarrow \C^m/\F^m_{-0.5} \rightarrow 0$}\label{section:SES2}

We now focus our attention to the following short exact sequence
\begin{equation}\label{eq:SES2}
0\rightarrow \F^m_{-0.5}/\F^m_{0} \stackrel{f}{\rightarrow} \C^m/\F^m_{0} \stackrel{g}{\rightarrow} \C^m/\F^m_{-0.5} \rightarrow 0.
\end{equation}

First, we consider the case when $m=2$. We show that sequence (\ref{eq:SES2}) splits, allowing us to characterize the quotient $\C^2/\F^2_0$.
\begin{prop}\label{prop:SES2}
The short exact sequence $0\rightarrow \F^2_{-0.5}/\F^2_{0} \stackrel{f}{\rightarrow} \C^2/\F^2_{0} \stackrel{g}{\rightarrow} \C^2/\F^2_{-0.5} \rightarrow 0$ splits.
\end{prop}

\begin{proof}
The group $\C^m/\F_{-0.5}^m \cong \Z^{m \choose 2},$ given by the pairwise linking numbers $lk(K_i,K_j)$ for $1 \le i < j \le m$ among the components of a representative string link.  Therefore when $m=2$, $\C^2/\F^2_0 \cong \Z$.  Since the map $g$ is a surjective map onto a free group, the universal property of free groups tells us that (\ref{eq:SES2}) splits.
\end{proof}

Using the proposition above, $\C^2/\F^2_0 \cong f(\F^2_{-0.5}/\F^2_0) \ltimes s(\C^2/\F_{-0.5}^2)$ where $s$ is a section $s:\C^2/\F^2_{-0.5} \rightarrow \C^2/\F^2_{0}$ given by the splitting.  

\begin{theorem}\label{thm:ab}
The quotient $\C^2/\F^2_0 $ is abelian.
\end{theorem}

\begin{proof}
The quotient $\C^2/\F^2_{-0.5}\cong \Z$ is generated by $T_1$, the Hopf string link. Consider $P,P' \in \C^2/\F^2_0$. Then $P$ can be written as $P=f(J) s(T_n)$ where $J\in \F^2_{-0.5}/\F^2_{0.5}$ and $T_n\in \C^2/\F^2_{-0.5}$ is the $n$-twisted Hopf string link. Similarly, $P'=f(J')s(T_{n'})$.  We show that the commutator $[P,P']$ is trivial in $\C^2/\F^2_0$.  

From the short exact sequence (\ref{eq:SES2}), notice that $f:\F^2_{-0.5}/\F^2_0 \rightarrow \C^2/\F^2_0$ is the inclusion map, so $f(J)=J$ and $f(J')=J'$. Further,  $\F^2_{-0.5}/\F^2_{0.5} \cong \Z_2^3$, so $J$ and $J'$ are of order two \cite{Martin}. The composition $g \circ s = id$, and thus we write $s(T_n)=T_n$ and $s(T_{n'})=T_{n'}$. Then, 

\begin{align*} 
[P,P']&= f(J)s(T_n)f(J')s(T_{n'})(f(J)s(T_n))\inv (f(J')s(T_{n'}))\inv \\
&=JT_n J' T_{n'} (J T_n))\inv (J'T_{n'})\inv \\
&=J T_n J' T_{n'} T_{-n} J^{-1} T_{-n'} J'^{-1}\\
&= J T_n J' T_{n'-n} J T_{-n'} J'
\end{align*}

As seen in Figure \ref{fig:commutator}, the twisting cancels and we see that $[P,P'] = [J,J'] = (JJ')^2$. 

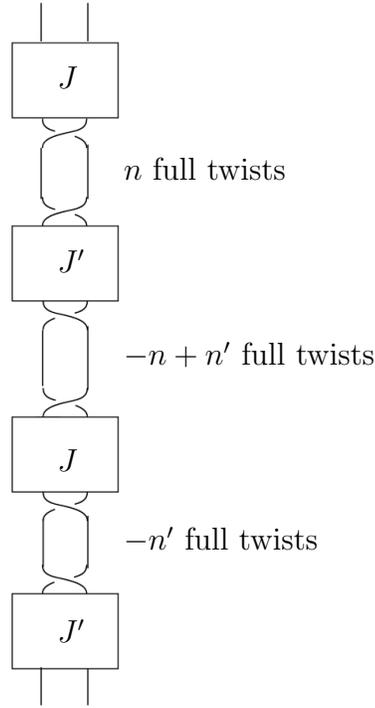
\begin{figure}
\centering
\includegraphics[scale=.6]{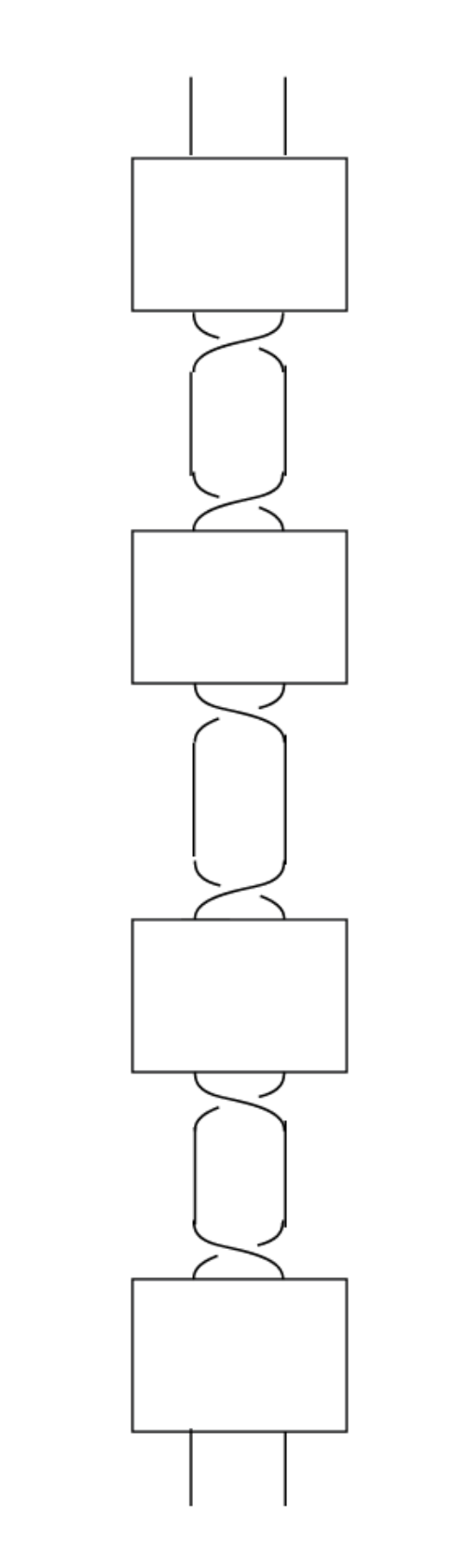}
\put(-45,245){$J$}
\put(-45,175){$J'$}
\put(-45,100){$J$}
\put(-45,35){$J'$}
\put(-20,210){$n$ full twists}
\put(-20,140){$-n+n'$ full twists}
\put(-20,70){$-n'$ full twists}
\caption{The commutator $[P,P']$}
\label{fig:commutator}
\end{figure}

This string link is 0-solvable if and only if the Arf invariant of each component is 0 and $\bar \mu _{(1122)}(L) \equiv 0$ mod 2 \cite{Martin}.  The Arf invariant is additive under the stacking operation, so the Arf invariants of any commutator are 0. Here, the Sato-Levine invariant is additive under stacking because it is the first non-vanishing Milnor's invariant. Thus, $\bar{\mu}_{(1122)}([P,P']) = \bar \mu_{(1122)}(J)+\bar \mu_{(1122)}(J')+\bar \mu_{(1122)}(J)+\bar \mu_{(1122)}(J')  \equiv 0 \bmod 2$. Hence, $[P,P']$ is 0-solvable and trivial in $\C^2/\F^2_0$.

\end{proof}

This gives the following corollary characterizing $\C^2/\F^2_0$.

\begin{corollary}\label{corollary:structure}
The quotient $\C^2/\F^2_0 \cong \Z_2 \oplus \Z_2 \oplus \Z_2 \oplus \Z$.
\end{corollary}

\begin{proof}
From Theorem \ref{thm:ab},  $\C^2/\F^2_0\cong f(\F^2_{-0.5}/\F^2_0) \x s(\C^2/\F_{-0.5}^2)$, where $f(\F^2_{-0.5}/\F^2_0) \cong \Z_2 \oplus \Z_2 \oplus \Z_2$ and $s(\C^2/\F_{-0.5}^2) \cong \Z$.
\end{proof}

However, for $m \ge 3$, $\C^m/\F^m_0$ is non abelian, as the Borromean Rings is an example of a nontrivial commutator in $\C^3/\F^3_0$ \cite{Otto}.  Furthermore, the sequence (\ref{eq:SES2}) does not split, as shown in the next theorem.

\begin{theorem}\label{thm:SES2nosplit}
For $m \ge 3$, the short exact sequence $0\rightarrow \F^m_{-0.5}/\F^m_{0}
{\rightarrow} \C^m/\F^m_{0} \stackrel{g}
{\rightarrow} \C^m/\F^m_{-0.5} \rightarrow 0$ does not split.
\end{theorem}

\begin{proof}
Suppose, to the contrary, that the short exact sequence (\ref{eq:SES2}) does split, and that there exists a section $s:\C^m/\F^m_{-0.5} \rightarrow \C^m/\F^m_0$ such that $g \circ s = id$. The quotient $\C^m/\F^m_{-0.5}$ is generated by $m \choose 2$ string links $S_{ij} = K_1 \cup \dots \cup K_m$ in which $lk(K_i, K_j) = 1$ and $lk(K_r,K_t) = 0$ for all $r,t \ne i,j$. 

Consider the image of these generators under the section. Let $L_{ij} = s(S_{ij}) \in \C^m/\F^m_0$. Necessarily, for each $i,j \ne r,t$, the commutators $[L_{ij},L_{rt}]$ are 0-solvable string links. Furthermore, from the splitting, we can write each $L_{ij} = K_{ij}A_{ij}$, where $K_{ij}$ is a string link with vanishing pairwise linking numbers and $A_{ij}$ is the standard $ij^{th}$ braid group generator. 

Then, we have that 
\begin{align*}
[L_{ij},L_{rt}]& =[K_{ij}A_{ij},K_{rt}A_{rt}]\\
& = [K_{ij},K_{rt}][K_{ij}A_{rt}]^{c_1}[A_{ij},K_{rt}]^{c_2}[A_{ij},A_{rt}]^{c_3},
\end{align*}  
where the $c_k$ denotes a conjugate of the commutator. As $[L_{ij},L_{rt}]$ is 0-solvable, we must have that for all $1 \le x <y<z \le m$, the triple linking numbers  $\bar{\mu}_{(xyz)}([L_{ij},L_{rt}]) = 0$. However, we can compute these Milnor's invariants to arrive at a contradiction. First, note that since the pairwise linking numbers of both $K_{ij}$ and $K_{rt}$ vanish,
\begin{align*}
\bar{\mu}_{(xyz)}([K_{ij},K_{rt}])&= \bar{\mu}_{(xyz)}(K_{ij}) + \bar{\mu}_{(xyz)}(K_{rt}) +\bar{\mu}_{(xyz)}(K_{ij}\inv ) + \bar{\mu}_{(xyz)}(K_{rt}\inv)\\
&= \bar{\mu}_{(xyz)}(K_{ij}) + \bar{\mu}_{(xyz)}(K_{rt}) -\bar{\mu}_{(xyz)}(K_{ij}) - \bar{\mu}_{(xyz)}(K_{rt})\\
&=0
\end{align*}

Next, we consider 
\begin{align*}
\bar{\mu}_{(xyz)}([K_{ij},A_{rt}]^{c_1})& = \bar{\mu}_{(xyz)}([K_{ij},A_{rt}])\\
&=\bar{\mu}_{(xyz)}(K_{ij}(A_{rt}K_{ij}\inv A_{rt}\inv))\\
&=\bar{\mu}_{(xyz)}(K_{ij})+\bar{\mu}_{(xyz)}(A_{rt}K_{ij}\inv A_{rt}\inv)
\end{align*}
as both of the string links $K_{ij}$ and $A_{rt}K_{ij}\inv A_{rt}\inv$ have vanishing pairwise linking numbers. However, the triple linking number of the string link $A_{rt}K_{ij}\inv A_{rt}\inv$ equals that of its closure;  \begin{align*}
\bar{\mu}_{(xyz)}(A_{rt}K_{ij}\inv A_{rt}\inv)&=\bar{\mu}_{(xyz)}(\widehat{A_{rt}K_{ij}\inv A_{rt}\inv})\\
&=\bar{\mu}_{(xyz)}(K_{ij}\inv) \\
&= -\bar{\mu}_{(xyz)}(K_{ij}).
\end{align*}
Therefore, $\bar{\mu}_{(xyz)}([K_{ij},A_{rt}]^{c_1})=\bar{\mu}_{(xyz)}(K_{ij})-\bar{\mu}_{(xyz)}(K_{ij})=0$. A similar argument shows that $\bar{\mu}_{(xyz)}([A_{ij},K_{rt}]^{c_2})=0$.

Thus, we see that $\bar{\mu}_{(xyz)}([L_{ij},L_{rt}]) = \bar{\mu}_{(xyz)}([A_{ij},A_{rt}])$. However, we cannot guarantee that $\bar{\mu}_{(xyz)}([A_{ij},A_{rt}])=0$. Note that the Borromean rings can be written as a commutator $[A_{12},A_{23}]$ in the braid group, and $\bar{\mu}_{(123)}([A_{12},A_{23}])= 1$. Thus, we cannot find such links $L_{ij}$ such that $[L_{ij},L_{rt}]$ is 0-solvable for all $i,j \ne r,t$. Thus, such a section $s:\C^m/\F^m_{-0.5} \rightarrow \C^m/\F^m_0$ does not exist and the short exact sequence (\ref{eq:SES2}) does not split.

\end{proof}

\bibliographystyle{amsalpha}
\bibliography{bibfile}

\end{document}